\documentclass[reqno, 11pt]{amsart}

\usepackage{amsfonts, amsthm, amsmath, amssymb}
\usepackage{hyperref}
\hypersetup{colorlinks=false}

\usepackage{bbm}  

 \usepackage{tabu}

\usepackage[margin=1in]{geometry}

\usepackage{helvet}

\RequirePackage{mathrsfs} \let\mathcal\mathscr

\usepackage{hyperref}
\usepackage{dsfont}
\usepackage{upgreek}
\usepackage{mathabx,yfonts}
\usepackage{enumerate}

\numberwithin{equation}{section}

\newtheorem{theorem}{Theorem}[section] 
\newtheorem{lemma}[theorem]{Lemma}
\newtheorem{proposition}[theorem]{Proposition}
\newtheorem{corollary}[theorem]{Corollary}

\theoremstyle{definition}

 \newtheorem*{acknowledgements}{Acknowledgements}
\newtheorem{remark}[theorem]{Remark}
\newtheorem{definition}[theorem]{Definition}

\newcommand{\w}{\widetilde}
\renewcommand{\phi}{\varphi}

\newcommand{\PP}{\mathbb{P}}

\newcommand{\FF}{\mathbb{F}}

\newcommand{\NN}{\mathbb{N}}
\newcommand{\QQ}{\mathbb{Q}}

\renewcommand{\leq}{\leqslant}

\renewcommand{\geq}{\geqslant}

\renewcommand{\c}{\mathbf{c}}

 \renewcommand{\b}{\mathbf{b}}

\renewcommand{\r}{\mathbf{r}}

\DeclareMathOperator{\Gal}{Gal}

\let\emptyset\varnothing

\DeclareSymbolFont{bbold}{U}{bbold}{m}{n}
\DeclareSymbolFontAlphabet{\mathbbold}{bbold}

\newcommand{\md}[1]{  \left(\textnormal{mod}\ #1\right)}
 
\renewcommand{\P}{\mathbb{P}}
\newcommand{\Q}{\mathbb{Q}}
\newcommand{\F}{\mathbb{F}}
\newcommand{\N}{\mathbb{N}}
\newcommand{\R}{\mathbb{R}}

\newcommand{\Z}{\mathbb{Z}}
\renewcommand{\l}{\left}

\renewcommand{\r}{\right}
\renewcommand{\b}{\mathbf}
\renewcommand{\c}{\mathcal}
\renewcommand{\epsilon}{\varepsilon}

\renewcommand{\leq}{\leqslant}
\renewcommand{\geq}{\geqslant}
\renewcommand{\#}{\sharp}

\newcommand{\beq}[2]
{
\begin{equation}
\label{#1}
{#2}
\end{equation}
}

\title[Gaps between prime divisors and analogues in Diophantine geometry]{Gaps between prime divisors and analogues in Diophantine geometry}

\author{Efthymios Sofos} 
\address{Department of Mathematics\\
University of Glasgow  \\ G12~8QQ United Kingdom}
\email{efthymios.sofos@glasgow.ac.uk}

\date{\today}
\begin{document}
\begin{abstract}Erd\H os considered the second moment 
 of  the gap-counting function 
of prime divisors in $1946$
and  proved an upper bound that is not of the right order of magnitude.
We prove  asymptotics for all moments. 
Furthermore, we prove a generalisation stating that the gaps between primes $p$ for which there is no $\Q_p$-point on a random variety 
are Poisson distributed.\end{abstract}

\maketitle

\setcounter{tocdepth}{1}
\tableofcontents 

\section{Introduction}   \label{s:intro} 
What are the typical gaps between the prime divisors of a randomly selected integer? For $m\in \N$
we let  $\omega(m)$ be the number of distinct prime divisors of $m$ and
 $p_i(m)$ be the $i$-th smallest prime divisor of $m$, so that 
\[ \log \log p_1(m)  < \ldots <  \log \log p_{\omega(m)}(m)  \] is a finite sequence  that depends on $m$.
It is not difficult to show  that for   almost all $m$ and  almost all $1\leq i\leq \omega(m)  $,
 one has  $\log \log p_i (m) \sim i $, hence,  $\log \log p_{i+1}(m)  -\log \log p_i (m) $ is typically   bounded. 
A natural question is to count the number of gaps exceeding a fixed constant $z\geq 0 $, i.e. estimate  
 \[\omega_z(m) := \#\l\{1\leq i < \omega(m) : \log \log p_{i+1}(m) - \log \log p_i(m)  > z\r\} .\]
Erd\H os~\cite[pg. 534]{MR0016078}  was the first to study this question. He showed that for almost all $m$ the  
function $\omega_z(m)$ is well-approximated by $\mathrm e ^{-z}\omega(m)  $
by proving an upper bound for the second moment:
\[\frac{1}{n} \sum_{m\in \N \cap [1,n]} \l(\omega_z(m)- \mathrm{e}^{-z} \log \log n  \r)^2
=o((\log \log n)^{3/2}), \text{ as } n\to+\infty.\] However, it turns out that this is not of the right order of magnitude.
Here we prove   asymptotics
not just for the second moment, but for all moments:
\begin{theorem}\label{thm: main theorem}
Fix any  $  z \geq 0 $ and  $r \geq 0 $.
Then  \[ \frac{1}{n} \sum_{m\in \N \cap [1,n]}
\l (\omega_z(m) - \frac{ \log \log n  }{ \mathrm{e}^z} \r)^r =\mu_r  
((1-2 z\mathrm{e}^{-z} ) \mathrm{e}^{-z} \log \log n)^{r/2}
(1+o(1)) , \text{ as } n\to+\infty,\]
where $\mu_r$ is the $r$-th moment of the standard normal distribution.
\end{theorem}
As a consequence, for all $\alpha<\beta \in \R$ one has  
\[
\lim_{n\to+\infty}\frac{1}{n}
 \#\l\{m \in \N\cap [1,n]: 
\frac{\omega_z(m)-\mathrm e ^{-z}\log \log m }{((1-2 z\mathrm{e}^{-z} ) \mathrm{e}^{-z} \log \log m)^{1/2}}
 \in (\alpha, \beta] \r\}  =\frac{1}{\sqrt{2\pi}} \int_\alpha^\beta \mathrm{e}^{-t^2/2} \mathrm{d} t .\] Setting $z=0$
 we recover the much celebrated 
Erd\H os--Kac theorem~\cite{MR0002374}. Our method is   different from that of 
Erd\H os~\cite{MR0016078}
in that it relies on Stein's method on normal approximations~\cite{MR882007}. 
This allows us to deal with certain sums of dependent random variables that arise when modelling  $\omega_z(m)$.
Stein's method  has  been rarely used in  number theory, for example, by Harper~\cite{MR2507311}.

 There are many generalisations of the Erd\H os--Kac theorem to functions of the form  $\sum_{p\mid m } g(p)$ but they do not cover $\omega_z(m)$, as $g(p)$ would  have to be a function of $m$ as well.
Galambos~\cite[Theorem 2]{MR1210525} studied the values of a function that is somewhat related to our $\omega_z$, namely the cardinality of $i < \omega(m)$ for which $\log \log p_{i+1}(m)  -\log \log p_i (m)> z +\log \log \log m $.
His results and method are rather different as they are suited to values of large gaps, while our result relates to small gaps. A function similar to Galambos' occurs in the recent work of 
Chan--Koymans--Milovic--Pagano~\cite[\S 4]{arXiv:1908.01752} on the negative Pell equation.

 \begin{remark}  
At the cost of a non-self sufficient argument,
 the number theoretic part  of the proof of Theorem~\ref{thm: main theorem}, (namely, Lemma~\ref{lem:moments})
can be alternatively verified via the Kubilius model~\cite[\S 12]{MR560507}. The approximation  of $\omega(m) $ by $ \mathrm e ^{-z}\log \log m $ 
means that the gaps in the sequence $\{\log \log p_i(m)\}_{i\geq 1 }$ are Poissonian. It is worth mentioning that the occurrence of 
Poisson distribution in other areas of Probabilistic 
Number Theory is not uncommon, see the work of 
de Koninck--Galambos~\cite{MR902529},
Harper~\cite{MR2507311},
Granville~\cite{MR2310491} and Kowalski--Nikeghbali~\cite{MR2725505}, for example. 
\end{remark}

 \begin{remark}[Further developments]
The interested reader may wonder whether one can use  tools from analysis 
 to make explicit the term  $o(1)$ in  Theorem~\ref{thm: main theorem}.
In the case of the Erd\H os--Kac theorem this was done by R\'{e}nyi and Tur\'{a}n \cite{MR96629} using complex analysis.
 After seeing the first version of this paper on arXiv, R. de la Bret\`eche and G. Tenenbaum 
  proved an explicit error term using methods quite different from
ours (namely, Fourier analysis); see their preprint~\cite{arXiv:1908.01752} for details.
\end{remark}
\subsection{Generalisations in Diophantine geometry}
In \S\ref{s:gener} we provide a generalisation of Theorem~\ref{thm: main theorem}, given by Theorem~\ref{thm2}. 
In brief terms it  states   that  the gaps between primes $p$ for which a typical variety over $\Q$ has no $\Q_p$-points
obey the Poisson distribution. A statement analogous to the Erd\H os--Kac theorem was proved by Loughran--Sofos~\cite{loughrsof}
by using geometric input from the work of
Loughran--Smeets~\cite{MR3568035}.

 \begin{acknowledgements} I wish to thank Daniel El-Baz for suggesting the use of Stein's method.
While working on this paper I was supported by EPSRC New Horizons grant \texttt{EP/V048236/1}.
I would like to thank Maxim Gerspach for various helpful remarks and for finding typos in the preprint version. 
Furthermore, I am  grateful to the referee for careful reading of the paper and helpful comments.
\end{acknowledgements}

\section{The proof of Theorem~\ref{thm: main theorem}} \label{s2}
\subsection{Defining the model} \hfill\\

The letter $z$ will denote a fixed non-negative real number throughout \S\ref{s2}.
As usual, we denote $\exp(z):=\mathrm e^z$.
For a prime  $p$ and a positive  integer $m$ we define \[
\delta_{p,z}(m) :=\begin{cases} \label{def:deltapz}
1, & \text{ if } p\mid m \text{ and } m \text{ is not divisible by any prime in } (p,p^{\exp(z) }],
\\  0, &  \text{ otherwise.} \end{cases}\] In particular, $\omega_z(m) = \sum_{p} \delta_{p,z}(m)$, where the sum is over all primes. Our plan, initially, is to follow
the Kubilius model idea (see Billingsley~\cite[Equations (1.8),(1.9)]{MR466055})
to define Bernoulli 
random variables $B_p$ that model the behavior of $\delta_{p,z}$.
For this we use the random variables $X_p$  as follows: for every prime $p$ the random variable $X_p$ is defined so that 
\[ P[X_p=1]=\frac{1}{p} , \hspace{1cm} P[X_p=0]=1-\frac{1}{p}   \] and such that $X_p$ are independent. 
In particular, the mean $E[X_p]$ equals $1/p$, thus, $X_p=1$ models   
the event that a random integer $m $ is divisible by a fixed prime $p$.
Let $[\cdot ]$ denote  the integer part. The independence of $X_p$ is related to the Chinese Remainder Theorem.

To model $\delta_{p,z}$ we must also take into account the fact that each prime $q$ in the range $(p,   p^{\exp(z)}]$ must not divide $m$.
Thus, we are naturally led to define
\beq{def:newrv}{B_p:=X_p \prod_{\substack{ q \textrm{ prime } \\ p<q\leq p^{\exp(z)} } }(1-X_q).}
We will later prove that $\sum_pB_p$ is a good  model for $\omega_z=\sum_p\delta_{p,z}$ in the sense that
their moments agree asymptotically.
\begin{remark}[Independence break-down]Definition~\eqref{def:newrv} leads to a major difference between this paper and the proofs of the Erd\H os--Kac theorem, 
namely, the variables $B_p$ are   dependent. Indeed, for all    primes $p<q$ with $q\leq p^{\exp(z)}$, the quantity  
$E[B_p B_{q}] $ vanishes while none of $E[B_p] , E[ B_{q}] $ does.
\end{remark}

\subsection{Distribution and moments of the model via Stein's method} \hfill\\

For any positive $N$ we define \[S_N=\sum_{p\leq N } B_p\] and denote its expectation and variance respectively by 
\[
c_N:=E\l[S_N\r]
\ \ 
\textrm{ and } 
\ \ 
s_N^2:=\textrm{Var}\l[S_N\r]
.\]
Our goal in this section is to prove that $(S_N-c_N)/s_N$ converges in law to the standard normal distribution as $N\to \infty $
and that its moments are asymptotically Gaussian. This will be done respectively in Propositions~\ref{prop:stein_lem} and~\ref{prop:moments_stein}. 
We   first need a few preparatory estimates.

\begin{lemma}\label{lem: first lem} We have \beq{eq:mertens product}{E[B_p]=\frac{\mathrm e ^{-z} }{p}+O\l( \frac{1}{p \log p }\r)  ,} 
 \beq{meanlem}{c_N=\mathrm e ^{-z}\log \log N+O(1)} and \beq{meanvar}{ s_N^2 =\l(1-\frac{2z}{ \mathrm{e}^z} \r) \frac{\log \log N}{\mathrm{e}^z} +O(1) .}
\end{lemma}
\begin{proof}
Recall that Mertens' theorem states that $\sum_{p\leq T} 1/p= \log \log T+c+O(1/\log T)$ for some constant $c$.
The independence of $X_p$ yields  \[E[B_p]=\frac{1}{p} \prod_{p<q\leq p^{\exp(z) }} \l(1-\frac{1}{q} \r) ,\] which, by 
the   approximation $1-\epsilon=\exp(-\epsilon+O(\epsilon^2)) $ for $| \epsilon| \leq 1 $  and Mertens' theorem is 
\[ \frac{1}{p}\exp\l(-\sum_{p<q\leq p^{\exp(z) }}\frac{1}{p} +O\l(\sum_{p<q\leq p^{\exp(z) }} \frac{1}{p^2}\r)\r)=
\frac{\exp (-\log \log p^{\exp(z)}+\log \log p +O(1/\log p ))}{p} 
.\]Since $\exp\l(O(1/\log p )\r)=1+O(1/\log p)$, this  is sufficient for~\eqref{eq:mertens product}. The estimate~\eqref{meanlem}  is directly deduced from it and the 
fact that $\sum_p (p \log p)^{-1} $ converges. Next, denoting $h_p=E[B_p]$  
we have  \[s_N^2= \sum_{p\leq  N} E\l[(B_p-h_p)^2\r] +2\sum_{p< q \leq  N} E\l[(B_p-h_p)(B_q-h_q) \r] .\]
First note that $ E\l[(B_p-h_p)^2\r]= E\l[B_p \r]-h_p^2=h_p(1-h_p) $.
Further, if $q>p^{\exp(z)}$ then  $B_p$ and $B_q$ are independent, hence,  $E\l[(B_p-h_p)(B_q-h_q) \r]=0$. 
If  $p<q\leq p^{\exp(z)}$ then 
$E[B_p B_q]$ vanishes, hence 
$$E\l[(B_p-h_p)(B_q-h_q) \r]=-E\l[B_p \r]h_q-h_pE\l[B_q  \r]+ h_p h_q=-h_ph_q .$$
We obtain \begin{align*}s_N^2&=
 \sum_{p\leq N} h_p(1-h_p) -2 \sum_{\substack{  p< q \leq  \min \{ N ,  p^{\exp(z)}  \}} } h_p h_q \\&=
c_N-\sum_{p\leq N} h_p^2   
-2 \sum_{\substack{ p \leq  N^{\exp(-z)} \\ p< q \leq   p^{\exp(z)}    } } h_p h_q
-2 \sum_{\substack{ N^{\exp(-z)}< p \leq N  \\ p< q \leq  N  } } h_p h_q
.\end{align*}By~\eqref{eq:mertens product} we have $h_p\ll 1/p$, hence,  $\sum_{p } h_p^2  =O(1)$ and 
\[ \sum_{\substack{ N^{\exp(-z)}< p \leq N  \\ p< q \leq  N  } } h_p h_q\ll\l( \sum_{ N^{\exp(-z)}< p \leq N  } \frac{1}{p}\r)^2 =O(1) .\] 
Hence,~\eqref{meanlem} gives \beq{eq:ihcab}{s_N^2=\mathrm e ^{-z}
\log \log N-2 \sum_{  p \leq  N^{\exp(-z)} }h_p \sum_{ p< q \leq   p^{\exp(z)}   } h_q
+O(1).}Using~\eqref{eq:mertens product} we see that 
\[
\sum_{  p \leq  N^{\exp(-z)} }h_p \sum_{ p< q \leq   p^{\exp(z)}   } h_q
=
\sum_{  p \leq  N^{\exp(-z)} }h_p \sum_{ p< q \leq   p^{\exp(z)}   } \l(\frac{\mathrm e ^{-z} }{q}+O\l(\frac{1}{q\log q }\r)\r)
,\]which, by Mertens' theorem and  $\sum_{q>t}(q\log q)^{-1} \ll (\log t )^{-1}$, equals 
\begin{align*}
\sum_{  p \leq  N^{\exp(-z)} }h_p \l(\frac{z}{\mathrm e ^z }+O\l(\frac{1}{\log p} \r) \r)
&=\sum_{  p \leq  N^{\exp(-z)} } \l(\frac{\mathrm e ^{-z} }{p}+O\l(\frac{1}{p\log p} \r)\r)  \l(\frac{z}{\mathrm e ^z }+O\l(\frac{1}{\log p} \r) \r)
\\&=\frac{z}{\mathrm e ^{2z} }\l(\sum_{  p \leq  N^{\exp(-z)} }  \frac{1}{p}\r)+O (1)
=\frac{z}{\mathrm e ^{2z} }(\log \log N)+O (1).\end{align*} Injecting this into~\eqref{eq:ihcab} concludes the proof.\end{proof}

\begin{lemma}\label{lem:678lem}  For all $u\in \N,  \b r\in \N^u$ 
and primes  $p_1,\ldots,p_u$  
we have \[ E\l[\prod_{i=1}^u |B_{p_i} - E[B_{p_i}]|^{r_i}  \r]
=O_{\b r }\l(\frac{1}{ \mathrm{rad} (p_1\cdots p_u )}\r),\]where $\mathrm{rad}$ denotes   the radical.  
\end{lemma}
\begin{proof}
We write the factorisation into prime powers of $\prod_{i=1}^ u p_i^{r_i} $
as $\prod_{j=1}^{v} q_j^{s_j}$, where $q_j$ are $v$ distinct primes.
This implies that 
 \[
E\l[\prod_{i=1}^u |B_{p_i} - E[B_{p_i}]|^{r_i}  \r]
=
E\l[\prod_{j=1}^v |B_{q_j} - E[B_{q_j}]|^{s_j}  \r]
.\]
Using $|B_{q_j} - E[B_{q_j}]|\leq B_{q_j} +E[B_{q_j}]\leq X_{q_j} +E[X_{q_j}]=X_{q_j} +1/{q_j}$
and 
 the binomial theorem yields
$$
|B_{q_j} - E[B_{q_j}]|^{s_j}
\leq 
\l(X_{q_j} +1/{q_j} \r)^{s_j}=\sum_{t_j\in [0,s_j]}  {s_j\choose t_j}
\frac{X_{q_j}^{t_j}  }{q_j^{s_j-t_j} }
  \ll_{\b s}\max_{t_j\in [0,s_j] }
 \frac{X_{q_j}^{t_j}  }{q_j^{s_j-t_j} }
,$$ hence, 
\[
E\l[\prod_{j=1}^v |B_{q_j} - E[B_{q_j}]|^{s_j}  \r]
  \ll_{\b s}
\max_{\b t \in [0,s_1] \times \cdots  \times  [0,s_v] }
 E\l[\prod_{j=1}^v \frac{X_{q_j}^{t_j}  }{q_j^{s_j-t_j} } \r]
.\]
By the   independence of the $X_q$ 
we infer  that 
\[
 E\l[\prod_{j=1}^v \frac{X_{q_j}^{t_j}  }{q_j^{s_j-t_j} } \r]
= \prod_{j=1}^v \frac{ E [ X_{q_j}^{t_j}  ] }{q_j^{s_j-t_j} } 
= 
\prod_{\substack{ j=1 \\ t_j =0 }}^v \frac{ 1    }{q_j^{s_j } } 
\prod_{\substack{ j=1 \\ t_j \geq 1  }}^v \frac{ E [ X_{q_j}   ] }{q_j^{s_j-t_j} } 
\leq 
\prod_{\substack{ j=1 \\ t_j =0 }}^v \frac{ 1    }{q_j  } 
\prod_{\substack{ j=1 \\ t_j \geq 1  }}^v  E [ X_{q_j}   ]  
=
\prod_{  j=1  }^v \frac{ 1    }{q_j  } 
. \]
The proof now concludes by noting that 
$\prod_{  j=1  }^v q_j   
$ is the radical of $\prod_{i=1}^ u p_i^{r_i} $.
\end{proof}
 The following lemma is the main tool in the proof of Theorem~\ref{thm: main theorem}. It is due to Stein~\cite[Corollary $2$, page $110$]{MR882007}.
\begin{lemma}
[Stein]
\label
{lem:stein}
Let $T$
be a finite set and
for each $t \in T$, let
$Z_t$ be
a real random 
variable and $T_t$
a subset of $T$
such that 
$E[Z_t]=0$,
$E[Z_t^4]<\infty$
and 
$E[\sum_{t\in T }Z_t  \sum_{s\in T_t}  Z_s]=1$.
Then for all real $b$,
\beq
{eq:the stein bound}
{
\l|
P\l[\sum_{t\in T} Z_t\leq b\r] - \frac{1}{\sqrt{2\pi}}\int_{-\infty}^b\mathrm{e}^{-t^2/2}\mathrm{d} t
\r|
\leq 
4 (\Psi_1
+ 
\Psi_2
+
\Psi_3
)
,}
where the terms $\Psi_i$ are defined through 
\[\Psi_1 
 = E\l[\sum_{t \in T}
\l|
E\l[Z_t | Z_s, s\notin T_t\r]
\r|
\r]
,
\Psi_2^2 
 =
 E\l[\sum_{t\in T}
|Z_t|
\l(\sum_{s\in T_t} Z_s \r)^2
\r]\] and\[
\Psi_3^2 
=
E\l[
\l\{
\sum_{t\in T}
\sum_{s\in T_t}
(Z_t Z_s-E[Z_s Z_t])\r\}^2\r].\]
\end{lemma}
\begin{proposition}\label{prop:stein_lem} Fix $z\geq 0 $ and $b \in \R$. 
For any   $N\in \N$ we have 
\[\l|P\l[ S_N\leq  c_N +bs_N\r]- \frac{1}{\sqrt{2\pi}}\int_{-\infty}^b\mathrm{e}^{-t^2/2}\mathrm{d} t\r|
\ll_z (\log \log N)^{-1/4} ,\] where the implied constant depends at most on $z$.
In particular,  $ (S_N- c_N)/s_N $ converges in law to the standard normal distribution as $N\to \infty$.\end{proposition}
\begin{proof}
We will  apply  
Lemma~\ref{lem:stein}
with \begin{itemize}\item
$T$ being the set of primes in
$[2,N]$,
 \item
$T_p$
being the set of primes in  
 $[p^{\exp(-z) },p^{\exp(z) }] \cap [2,N] 
$,  \item $
Z_p
= 
(B_p - E[B_p] )/s_N$ for 
$p\in T$.
\end{itemize}
Let  $Y_p:=B_p -E[B_p]$.
Note that if $q \notin T_p$
then $Z_q$ and $Z_p$ are independent, hence, $E[Y_pY_q]=0$.
Therefore,  \[s_N^2 =\sum_{p,q\leq N }E[Y_p Y_q  ] =\sum_{\substack{p\leq N \\ q\in T_p }}E[Y_p Y_q  ] ,\]
which verifies
$E[\sum_{p\in T }Z_p  \sum_{q\in T_p}  Z_q]=1$.
We next 
 observe that 
since for 
 every $q \notin T_p$
the random variables
$Z_q$ and $Z_p$ are independent,
one obtains
$
E\l[Z_p | Z_q, q\notin T_t\r]
=
E\l[Z_p  \r]
=
0,$
therefore \beq{eq:steinpsi1}{\Psi_1=0.}
 Next, we use Lemma~\ref{lem:678lem} to  obtain 
 \begin{align*}
\Psi_2^2  s_N^{3} =&
\sum_{\substack{ 
p\leq N ,
q\in T_p }}
 E\l[
| Y_p   | 
Y_{q}  ^2
\r]
+    2 \sum_{\substack{ p\leq N, q_1<q_2 \in T_p}} E\l[|
 Y_p  | 
 Y_{q_1}  
 Y_{q_2}  
\r]
\\
\ll 
&\sum_{\substack{ 
p\leq N, q\in T_p  
}} \frac{1}{p q}
+\sum_{\substack{ 
p\leq N, q_1,q_2 \in T_p  
}}
\frac{1}{p q_1 q_2 }
.\end{align*}
The sum  $\sum_{q\in  T_p }1/q $ is bounded only in terms of $z$ by Mertens' theorem.
It shows that  \beq{eq:steinpsi2}{ \Psi_2^2 \ll  s_N^{-3} \sum_{p\leq N} \frac{1}{p} \ll  (\log \log N)^{-1/2},} owing to~\eqref{meanvar}.

To bound $\Psi_3$ we write 
$
\c C_p
:=\sum_{q\in T_p}
\l(
Y_p 
Y_q 
-
E[  Y_p 
Y_q   ]
\r)
$
to  obtain  
\beq
{eq:final bound}
{
\Psi_3^2
s_N^{4} 
=
 \sum_{p \leq N }  
E\l[\c C_p^2 \r]
+
2
\sum_{p_1 < p_2 \leq N }
E\l[ 
\c C_{p_1}
\c C_{p_2}
\r]
.}
Furthermore, $E\l[\c C_p^2 \r]
$ can be written as 
 \[\sum_{q\in T_p}
E
\l[
\l(
Y_p 
Y_q 
-
E[ Y_p 
Y_q   ]
\r)^2
\r]
+
2
\sum_{q_1 <q_2 \in T_p }
E\l[
\l(
Y_p 
Y_{q_1} 
-
E[ Y_p 
Y_{q_1}   ]
\r)
\l(
Y_p 
Y_{q_2} 
-
E[  Y_p 
Y_{q_2}   ]
\r)
\r]
,\]
which can
be seen to be 
\[
\ll 
\sum_{q\in T_p}
\frac{1}{pq}
+\sum_{q_1 <q_2 \in T_p }
\frac{1}{pq_1 q_2}
\]
by Lemma~\ref{lem:678lem}. Alluding to
$ 
\sum_{q\in  T_p }1/q \ll  1 $
shows that 
\beq
{eq: proino ka8arisma}
{
 \sum_{p \leq N }  
E\l[\c C_p^2 \r]
\ll  \sum_{p\leq N} \frac{1}{p}
\ll \log \log N
.}
Let us now 
observe that 
if 
$p_2>p_1^{\exp(2z)}$
then 
$T_{p_1} \cap T_{p_2} =\emptyset $,
therefore 
$\c C_{p_1}$
and $
\c C_{p_2}
$ are independent.
Since for every   $p$ we have 
$E[\c C_p]=0$ by definition, 
 we get 
$
E\l[ 
\c C_{p_1}
\c C_{p_2}
\r]
=\prod_{i=1}^2
E\l[ 
\c C_{p_i}\r] =0$.
Thus,  \beq {eq: this means1} {\sum_{p_1 < p_2 \leq N } E\l[  \c C_{p_1} \c C_{p_2} \r]
= \sum_{\substack{  p_1 < p_2 \leq N  \\ p_2\leq p_1^{\exp(2z)} } } 
\sum_{\substack{ q_1\in T_{p_1} \\ q_2\in T_{p_2}}}  E\l[ \l( Y_{p_1}  Y_{q_1}  - E[ Y_{p_1}  Y_{q_1}   ] \r) \l( Y_{p_2} Y_{q_2}  - E[  Y_{p_2}  Y_{q_2}   ] \r) \r]
.}  By Lemma~\ref{lem:678lem} this is 
\[\ll  \sum_{  p_1  \leq N  } 
\sum_{p_1 < p_2 \leq p_1^{ \exp(2z)}} 
\sum_{  q_1\in T_{p_1}  }
\sum_{   q_2\in T_{p_2} }
  \frac{1}{ \textrm{rad}(p_1p_2 q_1 q_2) }. \]
For any positive integer $c$ and   prime $q$ we have $\textrm{rad}(c q)=\textrm{rad}(c)\frac{q}{\gcd(q,c)}$, hence, the sum over $q_2 $ is 
\[  \frac{1}{ \textrm{rad}( p_1 p_2 q_1 ) }
\sum_{\substack{ q_2\in T_{p_2} \\ q_2\in  \{p_1,p_2,q_1 \} } } 1
+
 \frac{1}{ \textrm{rad}( p_1 p_2 q_1) }
\sum_{\substack{ q_2 \in T_{p_2} \\ q_2\notin  \{p_1,p_2,q_1 \} } } \frac{1}{q_2 } 
\leq \frac{3+\sum_{q\in T_{p_2} } 1/q}{ \textrm{rad}( p_1 p_2 q_1) }
\ll_z  \frac{1}{ \textrm{rad}( p_1 p_2 q_1) }
\]  by Mertens' theorem. Hence, ~\eqref{eq: this means1} is 
 \[\ll 
\sum_{ \substack{  p_1  \leq N \\ p_1 < p_2 \leq p_1^{ \exp(2z)} }  }  
 \sum_{  q_1\in T_{p_1}  } \frac{1}{ \textrm{rad}( p_1 p_2 q_1) } 
=
\sum_{ \substack{  p_1  \leq N \\ p_1 < p_2 \leq p_1^{ \exp(2z)} }  }  
\frac{1}{ \textrm{rad}( p_1 p_2  ) } \l\{ 
\sum_{ \substack{  q_1\in T_{p_1} \\ q_1 \in \{p_1,p_2\} } } 1
+ 
\sum_{ \substack{  q_1\in T_{p_1} \\ q_1 \notin \{p_1,p_2\} } } \frac{1}{q_1}
\r\}
.\] The two sums over $q_1 $ in the right-hand side are both bounded only in terms of $z$. 
This can be proved similarly as before with the sum over $q_2$.
We obtain the bound 
\[
\ll 
\sum_{ \substack{  p_1  \leq N \\ p_1 < p_2 \leq p_1^{ \exp(2z)} }  }  
\frac{1}{ \textrm{rad}( p_1 p_2  ) }
=\sum_{  p_1  \leq N }\frac{1}{p_1 } 
\sum_{  p_1 < p_2 \leq p_1^{ \exp(2z)} }   \frac{1}{p_2 } 
\ll \sum_{  p_1  \leq N }\frac{1}{p_1 } \ll \log \log N.\] 
This shows that the quantity in~\eqref{eq: this means1} is $\ll \log \log N$, which, when combined with~\eqref{eq: proino ka8arisma},
can be fed into~\eqref{eq:final bound}
to yield $\Psi_3^2 s_N^{4} \ll \log \log N$. Invoking~\eqref{meanvar} provides us with 
$\Psi_3 \ll
1/\sqrt{\log \log N}$. Together with~\eqref{eq:steinpsi1}-\eqref{eq:steinpsi2} it implies that 
\[\l|P\l[ S_N\leq  c_N +bs_N\r]- \frac{1}{\sqrt{2\pi}}\int_{-\infty}^b\mathrm{e}^{-t^2/2}\mathrm{d} t\r|
 \ll_z (\log \log N)^{-1/4} \] owing to Stein's bound~\eqref{eq:the stein bound}. Finally,
letting $N\to \infty$ shows that $(S_N-c_N)/s_N$ converges in law to the standard normal distribution. 
\end{proof}
\begin{remark}
We next prove asymptotics for the moments of $(S_N-c_N)/s_N$. This is possibly the central proof in the present paper.
The argument is a modification of the one  by Billingsley~\cite[Lemma 3.2]{MR466055}, which relies on a version of 
the dominated convergence theorem. However, 
the underlying  random variables are now dependent, thus,  
we need to introduce the notion of \textit{linked indices}. \end{remark}
\begin{proposition}\label{prop:moments_stein} Fix $z\geq 0 $ and a positive integer  $r $. 
Then we have 
\[
\lim_{N\to \infty}
E\l[\l(\frac{S_N-c_N}{s_N}\r)^r   \r]
=\mu_r,\] where $\mu_r$ is the $r$-th moment of the standard normal distribution. 
\end{proposition}
\begin{proof} Take $2k$ to be  the least strictly positive integer  with  $r<2k $, 
so that   Proposition~\ref{prop:stein_lem} with~\cite[Example 2.21]{vaart}  implies  that  it suffices to   
prove that \[\sup_{N\geq 1 }\l|E\l[\l(\frac{S_N-c_N}{s_N}\r)^{2k }   \r]\r|\] is bounded only in terms of $k$ and $ z$. Equivalently, by~\eqref{meanvar} it suffices to show \[ E\l[\l( S_N-c_N \r)^{2k }   \r] = E\l[\l(\sum_{p\leq N} (B_p-E[B_p] )\r)^{2k} \r] \ll_{k,z} (\log \log N)^{k} .\]
The left side equals\[\sum_{u=1}^{2k }  \sum_{\substack{ \b r \in \N^u \\ 2k=\sum_{i=1}^u r_i }} \frac{(2k ) !}{r_1! \cdots r_u! } 
\sum_{p_1<\ldots < p_u\leq N} E\l[\prod_{i=1}^u (B_p-E[B_p])^{r_i}  \r].\] Using Lemma~\ref{lem:678lem} we see that the contribution 
of the terms with $ u \leq k  $ is \[\ll_k\max_{1\leq u \leq k }\l(\sum_{p\leq N}\frac{1}{p}\r)^u
\ll_k (\log \log N)^{k }.\] Therefore, \beq{eq:walcha}{E\l[\l( S_N-c_N \r)^{2k }   \r] \ll \max_{\substack{u\in [k+1,2k] \\ \b r \in \N^u:=\sum_i r_i =2k } }
 \sum_{p_1<\ldots < p_u\leq N} E\l[\prod_{i=1}^u (B_p-E[B_p])^{r_i}  \r]+ (\log \log N)^{k} ,} with an implied constant that is independent of $N$.

For given $u \in \N $, $z\geq 0 $  and primes $p_1<\ldots < p_u$  we say that two   consecutive  integers $ i, i+1$  in $[1,u]$ 
are \textit{linked}  if and only if $p_{i+1}\leq p_{i}^{\exp(z)}$. In particular, $p_{i+1}$   lies in a relatively small interval, hence, its contribution will be small. 
Denote the number of  linked pairs $(i,i+1)$ by $\ell(\b p ) $. By  Lemma~\ref{lem:678lem} we obtain 
 \[\sum_{p_1<\ldots < p_u\leq N} E\l[\prod_{i=1}^u (B_p-E[B_p])^{r_i}  \r]\ll_z\l( \sum_{p\leq N} \frac{1}{p}\r)^{u-\ell(\b p ) }\ll (\log \log N)^{u-\ell(\b p ) },\]
where we used  the estimate $\sum_{p_i<p_{i+1} <p_i^{\exp(z)}}1/p_i \ll_z 1 $ whenever $i$ and $i+1$ are linked.
Hence, the contribution of all prime vectors $(p_1,\ldots, p_u)$
 with at least $ \ell(\b p )\geq  u-k   $  linked   pairs  is at most 
\[ \ll    (\log \log N)^{u-\ell(\b p ) } \ll (\log \log N)^{ k  },\] which is acceptable.
By~\eqref{eq:walcha} we obtain 
\beq{eq:walcha2}{E\l[\l( S_N-c_N \r)^{2k }   \r] \ll \max_{\substack{u\in [k+1,2k] \\ \b r \in \N^u:=\sum_i r_i =2k } }
\
 \sum_{\substack{ p_1<\ldots < p_u\leq N\\ \ell(\b p ) <   u-k } } E\l[\prod_{i=1}^u (B_p-E[B_p])^{r_i}  \r]+ (\log \log N)^{k} ,} 
We will now show that every sum over $p_i$ 
in~\eqref{eq:walcha2} vanishes.   
Denoting the cardinality  of $1\leq i\leq u $ with  $r_i=1$ by $a$ 
we see that  
the number of $i$ with $r_i \geq 2$ is $u-a$.
Since $2k =\sum_{i=1}^u r_i$ we get  $2k \geq a+ 2(u-a)$.
Equivalently, $ 2 (u-k)\leq a $,
hence, by $ \ell(\b p ) <   u-k$
one gets  \beq{eq: number of ones}{2\ell(\b p ) < \#\{ i \in [1,u]: r_i=1 \} .}We now partition the integers in $ [1,u]$
into   disjoint subsets $\c A_1,\ldots,\c A_r$ using the following rules:
\begin{itemize}\item if $i$ and $i+1$ are in $S_j$ then they are linked,
\item if $i \in S_a$ and $i+1 \in S_b$ for some $a\neq b$ 
then $i$ and $i+1$ are not linked.\end{itemize}
The   inequality $s\leq 2(-1+s)$ for   $s\geq 2$ 
gives
\[
\#\{ i \in [1,u]: i \textrm{ linked to some index}\}=
\sum_{\substack{1\leq j \leq r \\  
2\leq \#\c A_j}} \#\c A_j
\leq 
\sum_{\substack{1\leq j \leq r \\  
2\leq \#\c A_j}} 
2(-1+\#\c A_j ) .
\] This equals $2\ell (\b p ) $ since each   $\c A_j$ has $-1+\#\c A_j$ linked pairs and  the total number of links is $\ell(\b p ) $.
By~\eqref{eq: number of ones} we infer that  there exists an index $j $ for which $r_j=1$
and that is not linked to any other index. This implies
that    the following random variables are independent:
\[
 \prod_{\substack{ 1\leq i \leq u \\ i \neq j } }
(B_{p_i} - E[B_{p_i}] )^{r_i}  
\ \ \ \text{  and   } \ \ \ 
(B_{p_j} - E[B_{p_j}] )^{r_j}   = B_{p_j} - E[B_{p_j}] .\] Since $ E\l[  B_{p_j} - E[B_{p_j}]   \r]=0$
we infer that every expectation in the right hand side of~\eqref{eq:walcha2} vanishes. This concludes the proof. 
\end{proof}
\subsection{Justifying the model} \label{s:justif}\hfill\\

Let $n $ be a positive integer and denote by $\Omega_n$ the uniform probability space $\N\cap [1,n]$.
Our goal now becomes to show that, as $n\to\infty $,
the moments of $\omega_z(m)$ for $m$ in $\Omega_n$
are asymptotically the same as the moments of $S_N$ for some parameter $N=N(n)\to\infty$.
Recall~\eqref{def:deltapz}. For technical reasons we will first work with a truncated version of $\omega_z$, namely,
\beq{def:walc}{\omega_{z,N}(m)=\sum_{p\leq N}\delta_{p,z}(m),} where $N=N(n)$.
The function $\delta_{p,z}$ imposes simultaneous coprimality conditions of $m$ with several primes in large intervals and to deal with this
we shall need the Fundamental Lemma of Sieve Theory~\cite[Corollary 6.10]{MR2647984}.
\begin{lemma}[Fundamental Lemma of Sieve Theory]\label{lem:fund}
Let $\c P$ be a set of primes. Given any    sequence $a_m\geq 0 $ for $m \in \N$
 and any square-free $d\leq x $ that is only divisible by primes in $\c P$
we assume that
\[ \sum_{m\leq x \atop m\equiv 0 \md d } a_m=X g(d)+r_d \] for some real numbers $X, r_d$
and  a multiplicative function $g$.
Assume that $0\leq g (p)<1 $
and   that there exist  constants $K>1, \kappa>0$ such that  
$$ \prod_{w\leq p < y\atop p\in\c P } (1-g(p))^{-1}\leq K \left(\frac{\log y}{\log w } \right)^\kappa$$ holds for all $2\leq w < y $.
Then for  all   $D\geq y \geq 2 $ we have 
\beq{eq:givv}{\sum_{m\leq x \atop p\in\c P, p<y 
 \Rightarrow p\nmid m  }a_m=X \left(\prod_{p<y  \atop p\in\c P}(1-g(p))\right)   \{1+O(\mathrm e^{-s})\}+O\left(\sum_{ d < D \atop 
p\mid d \Rightarrow p\in \c P}\mu^2(d) |r_d |\right) ,} where $s=\log D/\log y$ and the implied constants depend at most on $\kappa$ and $K$.
\end{lemma} 
\begin{lemma} \label{lem:moments} 
Assume that there exists a   function $N:[1,\infty)\to [1,\infty)$ satisfying \begin{align}
&\lim_{n\to\infty}N(n)=+\infty,  \label{eq:assumption1} \\
&\limsup_{n\to\infty}\frac{(\log N(n))(\log \log \log N(n))}{\log n }\neq +\infty.  \label{eq:assumption2}  
\end{align}Fix $z\geq 0 $ and  $k\in \N$. Then  we have 
\[\lim_{n\to\infty}  E_{m\in \Omega_n}\l[\l(\frac{\omega_{z,N}-c_N}{s_N}\r)^k\r] =\mu_k,\]
where $\mu_k $ is the $k$-th moment of the standard normal distribution.\end{lemma}
\begin{proof}
By Proposition~\ref{prop:moments_stein} and~\eqref{eq:assumption1} 
it is sufficient to prove 
\beq{eq:suff}{ \lim_{n\to\infty} \l(   E_{m\in \Omega_n}\l[\l(\frac{\omega_{z,N}(m)-c_N}{s_N}\r)^k\r]  -   E \l[\l(\frac{S_N-c_N}{s_N}\r)^k\r]  \r) =0.}
Let $r\in \N$. By~\eqref{def:walc}, the fact that  $\delta_{p,z}\in \{0,1\}$ and the binomial theorem we obtain  
\beq{eq:binomi2}{E_{ m\in \Omega_n}\l[\omega_{z,N}(m)^r \r] =\sum_{u=1}^r \sum_{\substack{ r_1,\ldots, r_u \in \N \\ r_1+\ldots +r_u=r }} \frac{r!}{r_1!\cdots r_u!} \sum_{p_1 < \cdots <p_u \leq N } E_{m\in   \Omega_n} \l[ \delta_{p_1,z}(m) \cdots \delta_{p_u,z}(m) \r] .}
Let $\c P$ be the set of all primes in  $\bigcup_{i=1}^u (p_i, p_i^{\exp(z)}]$  and let $a_m $ be the indicator function of integers divisible by $p_1\cdots p_u$.
In particular,  $$E_{m\in \Omega_n} \l[ \delta_{p_1,z}(m) \cdots \delta_{p_u,z}(m) \r] =\frac{1}{n}
\sum_{\substack{ 1\leq m \leq n\\ p\in \c P  \Rightarrow p\nmid m  }  }a_m  .$$ 
We assume that  $p_{i+1}>p_i^{\exp(z)}$ for all $i=1,2,\ldots, u-1$ since otherwise the sum clearly vanishes.
We will now    use   Lemma~\ref{lem:fund}
with $ X=n/(p_1\cdots p_u), g(d)=1/d, D=\sqrt n, y= N^{2\exp(z)}.$
If $d $ is divisible only by primes in $\c P$ then it is coprime to $p_1\cdots p_u$, hence, 
\[ \sum_{m\leq n \atop m\equiv 0 \md d } a_m  = \l[\frac{n}{p_1\cdots p_u d }\r],\] thus, $|r_d|\leq 1 $ because 
$ r_d $ is the fractional part of $X/d$. Furthermore, we can take $K$ to be any large fixed positive constant and $\kappa =1$, 
owing to \[\prod_{w\leq p < y \atop p\in \c P } (1-g(p))^{-1} =\prod_{w\leq p < y \atop p\in \c P } (1-1/p)^{-1} 
\leq \prod_{w\leq p < y   } (1-1/p)^{-1} \ll \frac{\log y}{\log w}.\]
The bound $|r_d|\leq 1 $,   means that $\sum_{d\leq D} \mu^2(d) |r_d| \leq D=\sqrt n$.
Since $p_u\leq N$, every prime in $\c P$ is strictly smaller than     $y$, hence,~\eqref{eq:givv}
gives  \beq{eq:sivres}{E_{m\in \Omega_n} \l[ \delta_{p_1} \cdots \delta_{p_u} \r] = \l\{ \prod_{i=1}^u \frac{1}{p_i} \prod_{p_i<q\leq p_i^{\exp(z)}}
\left(  1-1/p \right)\r\} \l\{1+O\l(\mathrm e^{-\frac{\log   n}{4\exp(z)\log N}}\r)\r\} +O(n^{-1/2}),} where the implied constant depends at most on $r$ and $z$.

By the binomial theorem we get 
\[ E\l[S_N^r \r]=E\l[\l(\sum_{p\leq N}B_p \r)^r \r]
=\sum_{u=1}^r \sum_{\substack{ r_1,\ldots, r_u \in \N \\ r_1+\ldots +r_u=r }} \frac{r!}{r_1!\cdots r_u!} \sum_{p_1 < \cdots <p_u \leq N } 
E  \l[ B_{p_1} \cdots B_{p_u} \r]\] and we note that we can restrict the sum over $p_i$ to the terms with  $p_{i+1}>p_i^{\exp(z)}$ for all $i$, 
since otherwise $E  \l[ B_{p_1} \cdots B_{p_u} \r]=0$. Under this restriction the random variables $B_{p_i}$ are   independent, hence, 
\[  \prod_{i=1}^u \frac{1}{p_i} \prod_{p_i<q\leq p_i^{\exp(z)}}\left(  1-1/p \right) = E  \l[ B_{p_1} \cdots B_{p_u} \r].\]
We infer from~\eqref{eq:binomi2} and~\eqref{eq:sivres} that  \[\l| E_{ m\in \Omega_n}\l[\omega_{z,N}(m)^r \r]-E\l[S_N^r \r] \r|
\ll_r E\l[S_N^r \r] \mathrm e^{-\frac{\log   n}{4\exp(z)\log N}}+n^{-1/2} \sum_{u=1}^r\sum_{p_1<\ldots < p_u\leq N }1 .\]
By~\eqref{meanlem} this is $\ll(\log \log N)^r \exp(-\frac{\log   n}{4\exp(z)\log N})+n^{-1/2}N^r$.
Thus, the difference in~\eqref{eq:suff} is 
\begin{align*}
\ll&s_N^{-k}\sum_{r=0}^k {k\choose r} (-c_N)^{k-r}\l( E_{   \Omega_n}\l[\omega_{z,N} ^r \r]-E\l[S_N^r \r] \r)
\\ \ll  
&s_N^{-k} 
\l\{\mathrm e^{-\frac{\log   n}{4\exp(z)\log N}} (c_N+\log \log N)^k
 +n^{-1/2} (N+c_N)^k  
\r\}.\end{align*} We need to show that this vanishes asymptotically and by~\eqref{meanvar} and~\eqref{eq:assumption1}
it suffices to show 
\[
 (2\log \log N)^k\leq \exp\l( \frac{\log   n}{4\exp(z)\log N}\r)
\ \ \textrm{ and } \ \ 
(2N)^k \leq n^{1/2}.\] Both of these inequalities can be directly inferred from~\eqref{eq:assumption1}-\eqref{eq:assumption2}.
\end{proof}\begin{lemma} \label{lem:mnts23} 
Assume that there exists a   function $N:[1,\infty)\to [1,\infty)$ satisfying \begin{align}
&\lim_{n\to\infty} \frac{\log \log N(n)}{\log \log n }=1,  \label{eq:assumption3} \\
&\limsup_{n\to\infty}\frac{(\log N(n))\sqrt{ \log \log n}}{\log n }= +\infty.  \label{eq:assumption4}  
\end{align}Fix $z\geq 0 $. Then 
${s_N} ( (1-2z\mathrm e ^{-z}) {\mathrm e }^{-z}\log \log n )^{-1/2}  \to 1$ as $n\to \infty $  and  
\[\lim_{n\to \infty} \frac{ \max \l\{  \l|   ( \omega_{z}-\mathrm e ^{-z}\log \log n ) - ( \omega_{z,N}-c_N  )   \r| :m \in \N \cap[1,n] \r\} } {\sqrt{\log \log n } }=0 .\]
\end{lemma}
\begin{proof}Combining~\eqref{meanvar} and~\eqref{eq:assumption3}  one immediately gets 
$$\lim_{n\to\infty } \frac{s_N}{ ( (1-2z\mathrm e ^{-z}) {\mathrm e }^{-z}\log \log n )^{1/2} }  = 1. $$
For any  $m \in  [1,n]$ we have \[ \l|   ( \omega_{z}-\mathrm e ^{-z}\log \log n ) - ( \omega_{z,N}-c_N  )   \r|\leq 
\sum_{p>N}\delta_{p,z}(m)+|\mathrm e ^{-z}(\log \log n)-c_N | . \] Since $\delta_{p,z}$ takes only values in $\{0,1\}$
and $\delta_{p,z}(m)=1$ implies that $p$ divides $m$,we see that 
\[\sum_{p>N}\delta_{p,z}(m) \leq \#\{p\mid m : p>N\} \leq \frac{\log m}{\log N}  \leq \frac{\log n}{\log N} .\]
Furthermore,~\eqref{meanlem} gives \[ \mathrm e ^{-z}(\log \log n)-c_N \ll 1+ \log\frac{\log n }{\log N}\ll \frac{\log n }{\log N}.\]
The proof now concludes by using~\eqref{eq:assumption4}.
\end{proof}
\subsection{Proof of Theorem~\ref{thm: main theorem}}\label{prfmainthm}
 The function $$N(n):= n^{1/\log \log n}$$ fulfills~\eqref{eq:assumption1}-\eqref{eq:assumption2}-\eqref{eq:assumption3}-\eqref{eq:assumption4},
hence, we can apply Lemmas~\ref{lem:moments}-\ref{lem:mnts23}.

For any $r\in \N$, $c\in \mathbb C$ any probability space $\Omega_n$ and any two sequences of random variables $X_n,Y_n$ satisfying  
 $\lim_{n\to\infty } \sup_{m \in \Omega_n}|X_n(m)-Y_n(m) |=0$ and $\lim_{n\to \infty} E_{m\in \Omega_n} [X_n(m)^r]=c$
it is easy to see by the binomial theorem that  $\lim_{n\to \infty} E_{m\in \Omega_n}[Y_n(m)^r]=c$.
 Using this with $ \Omega_n=\N\cap[1,n] $, \[X_n(m)=  \frac{\omega_{z,N}(m)-c_N}{s_N} 
\  \ \textrm{ and } \ \ 
Y_n(m)=\frac{\omega_{z}(m)-\mathrm e ^{-z}\log \log n}{s_N}
,\] in combination with  Lemmas~\ref{lem:moments}-\ref{lem:mnts23}, shows that for every $k \in \N $ one has 
\[\lim_{n\to\infty}  E_{m\in \Omega_n}\l[\l(\frac{\omega_{z}(m)-\mathrm e ^{-z}\log \log n}{s_N}\r)^k\r] =\mu_k.\]
Given any sequence $a_n \in \R$ having limit $1$ and any sequence of random variables $X_n$ 
with $E[X_n] $ having limit $c$ it is clear  that $a_n E[X_n]$ has limit $c$.
Using this with 
\[a_n= \frac{s_N}{ ( (1-2z\mathrm e ^{-z}) {\mathrm e }^{-z}\log \log n )^{1/2} } 
\  \ \textrm{ and } \ \ 
X_n(m)=\frac{\omega_{z}(m)-\mathrm e ^{-z}\log \log n}{s_N}
\] and invoking Lemma~\ref{lem:mnts23}
shows that for every $k\in \N$ one has 
\beq{eq:suffc}{\lim_{n\to\infty}  E_{m\in \Omega_n}\l[\l(\frac{\omega_{z}(m)-\mathrm e ^{-z}\log \log n}{ ( (1-2z\mathrm e ^{-z}) {\mathrm e }^{-z}\log \log n )^{1/2} }\r)^k\r] =\mu_k .} This proves Theorem~\ref{thm: main theorem} whenever $r$ is a positive integer and this is  sufficient. To see that, take any $r\in [0,\infty)$
and note that~\eqref{eq:suffc} implies that \[ T_n=\frac{\omega_{z}(m)-\mathrm e ^{-z}\log \log n}{ ( (1-2z\mathrm e ^{-z}) {\mathrm e }^{-z}\log \log n )^{1/2} }\]
converges in law to the standard normal distribution. Taking $p$ to be the least even integer strictly exceeding $r$ 
in~\cite[Example 2.21]{vaart} shows that the $r$-th moment of $T_n$ converges to the $r$-th moment of the standard normal distribution. 
\qed

\section{Poissonian gaps   for local solubility in families of varieties}\label{s:gener}
Serre's problem~\cite{MR1075658}
on the probability that a random variety over $\Q$ has a $\Q$-rational point has 
recently received a lot of attention due to its extension  by Loughran--Smeets~\cite{MR3568035} to a very general setting, namely, 
for any dominant morphism $f: V \to \PP^n$, where,  $V$ is a smooth projective variety over $\Q$ and $f$ has 
a   geometrically integral generic fibre. The fibres of $f$ form an infinite family of varieties
and typically one is interested in how often they have a $\Q$-rational point.
Imposing the harmless condition that the generic fibre of $f$ is geometrically integral, it is easy to see that  
for  every   $x$ outside of some proper Zariski closed set
the function \[\omega_f(x) := \#\l\{ p \textrm{ prime}: (f^{-1}(x))(\Q_p)=\emptyset \r\} ,\] is bounded
due to the Lang--Weil estimates   and Hensel's lemma. This function 
helps us in understanding   the density of fibres with a  $\Q$-rational point.
Ordering $\P^n(\Q)$ by   
the standard Weil height $H$ on $\PP^n(\QQ)$ and assuming that a certain   invariant $\Delta(\pi)$ 
is non-vanishing, 
Loughran and Sofos~\cite{loughrsof} recently
proved the analogue of Erd\H os--Kac's theorem for  $\omega_f(x) $,
namely that  $$\frac{\omega_f(x)-\Delta(\pi) \log \log H(x) }{(\Delta(\pi) \log \log H(x) )^{1/2} }$$
converges in law to the standard normal distribution.  This was     the first instance of an Erd\H os--Kac law in Diophantine geometry. 

Our goal in this section is to go further and study the gaps between the primes $p$  counted by $\omega_f(x)$.
For   $x\in \PP^n(\QQ)$   with $f^{-1}(x)$   smooth we let $p_i(x)$ be the $i$-th smallest prime number for which $f^{-1}(x)$ has no $\Q_p$-point.
We then define for all $z\geq 0 $, $$\omega_{f,z}(x):=\#\{i\geq 1 : \log \log p_{i+1}(x) - \log \log p_i(x) > z\}.$$
Before stating our theorem we must recall the definition of  the invariant $\Delta(f)$ that is due to Loughran and Smeets~\cite{MR3568035}.
 \begin{definition} \label{def:Delta}
	Let $f:V \to X$ be a dominant proper morphism of smooth irreducible varieties over a
	field $k$ of characteristic $0$. For each  point
	$x \in X$ with residue field $\kappa(x)$,
	the absolute Galois group $\Gal(\overline{\kappa(x)}/ \kappa(x))$ 
	of the residue field acts on the irreducible
	components of $$f^{-1}(x)_{\overline{\kappa(x)}}:=f^{-1}(x) \times_{\kappa(x)} \overline{\kappa(x)}$$ of multiplicity $1$. 
	Choose	some finite group $\Gamma_x$ through which this action factors and define
	$$\delta_x(f) = \frac{\# \left\{ \gamma \in \Gamma_x : 
	\begin{array}{l}		\gamma \text{ fixes an irreducible component} \\	\text{of $f^{-1}(x)_{\overline{\kappa(x)}}$ of multiplicity } 1
	\end{array}	\right \}}	{\# \Gamma_x }	$$  and  $$ \Delta(f) = \sum_{D \in X^{(1)}} ( 1 - \delta_D(f)),$$
	where $X^{(1)}$ denotes the set of codimension $1$ points of $X$. \end{definition}

\begin{theorem}\label{thm2}Let $V$ be a smooth  projective variety over $\QQ$ equipped with a dominant morphism $f: V \to \PP^n$ 
with geometrically integral generic fibre and $\Delta(f)\neq 0$. Let $H$ be the usual Weil height on $\P^n$. 
Fix any  $  z \geq 0 $ and  $r \geq 0 $.
Then \begin{align*}
 \sum_{\substack{ x \in \PP^n(\QQ), H(x)\leq B\\ f^{-1}(x) \text{ smooth}}}
&\l(\frac{\omega_{f,z}(x)-\Delta(f)\exp( -z\Delta(f) ) \log \log B}{\sqrt{\Delta(f)\exp( -z\Delta(f) )\log \log B}}\r)^{r}
  \\  =& \mu_r   \l(1-\frac{2\Delta(f)z}{ \exp(\Delta(f)z)}  \r)^{r/2}
\#\l\{ x \in \PP^n(\Q) : H(x)\leq B\r\}(1+o(1)) , \end{align*} as $B\to \infty $,  where $\mu_r$ is the $r$-th moment of the standard normal distribution.\end{theorem}
 The case $z=0$   recovers Theorems 1.2-1.3 of Loughran--Sofos~\cite{loughrsof}.

Taking $r=2$ in Theorem~\ref{thm2} and~\cite[Theorem 1.2]{loughrsof} shows the following after a use of Chebychev's inequality: 
 \begin{corollary}\label{corol45} Let $f:V\to \P^n$ be a  morphism as in Theorem~\ref{thm2}. Fix any $z\geq 0 $.
Ordering $\P^n(\Q)$ by the usual Weil height, $100\%$ of fibres $f^{-1}(x)$ satisfy 
\[\l|\frac{\omega_{f,z}(x)}{\omega_f(x)}- \frac{\Delta(f)}{\mathrm e ^{ z\Delta(f) } }\r| \leq (\log \log H(x))^{-1/4}.\]\end{corollary}
\begin{remark}As the right hand side vanishes asymptotically, the corollary means  that for almost all fibres $f^{-1}(x)$,
the proportion of  gaps   in the sequence $\{\log \log p_i(x)\}_{i\geq 1 }$  exceeding $z$ is roughly constant, independently of the 
fibre!
\end{remark}

In our proof we use the arguments from \S\ref{s2}, 
where the uniform probability space $\N \cap [1,n]$ is replaced by $\{x\in \P^n(\Q):H(x) \leq B\}$.
The main number-theoretic we use is Proposition~\ref{warcturus}.
In sieve theory language this is  a level of distribution result for the fibres of $f$. 
The level of distribution it provides is less than $B^\epsilon$ for any constant $\epsilon>0$, which
is   well-known to be a problematic regime for   any sieve theory problem;
we overcome this  by extirpating   small primes $p\leq t_0(B)$ from  $\w \omega_{f,z}$, see~\eqref{wwwalc}.
\subsection{Proof of Theorem~\ref{thm2}}
For a prime $p$ we define\[\sigma_p:=\frac{\#\big\{x \in \P^n(\F_{p}): f^{-1}(x) \mbox{ is non-split}\big\}}{\#\PP^n(\FF_p)},
\]where  we use the term ``non-split'' in the sense of Skorobogatov \cite[Def.~0.1]{skor}. We then introduce the random variable
$\w{X}_p$ so that  \[ P[\widetilde X_p=1]=\sigma_p , \hspace{1cm} P[\widetilde X_p=0]=1-\sigma_p   \] and such that $\widetilde X_p$ are independent. 
We then define\[\widetilde B_p:=\widetilde X_p \prod_{\substack{ q \textrm{ prime } \\ p<q\leq p^{\exp(z)} } }(1-\widetilde X_q).\]
Furthermore, for any positive $N$ we define \[\widetilde S_N=\sum_{p\leq N } \widetilde B_p, \ \ \ \
\widetilde c_N:=E\l[\widetilde S_N\r]\ \ \textrm{ and } \ \ \widetilde s_N^2:=\textrm{Var}\l[\widetilde S_N\r].\]
Using~\cite[Proposition 3.6]{loughrsof} instead of Mertens' theorem
and the estimate $\sigma_p \ll 1/p$ from~\cite[Lemma 3.3]{loughrsof}, the arguments in Lemma~\ref{lem: first lem} 
 can be modified to yield
 \beq{wmertens product}{E[\w B_p]=\exp( -z\Delta(f) ) \sigma_p +O\l( \frac{1}{p \log p }\r)  ,} 
 \beq{wwmeanlem}{\w c_N=\Delta(f)\exp( -z\Delta(f) )\log \log N+O(1)} and 
\beq{wmeanvar}{ \w s_N^2 =\l(1-\frac{2\Delta(f)z}{ \exp(\Delta(f)z)} \r) \frac{\Delta(f)\log \log N}{ \exp(\Delta(f)z)} +O(1) .}
Next, the proof of Lemma~\ref{lem:678lem} goes through easily upon replacing $B_p $ by $\w B_p$ 
owing to the inequality $E[\w B_p]\leq E[\w X_p] =\sigma_p \ll 1/p$. Replacing $S_N$ by $\w S_N$ in the statement of Proposition~\ref{prop:stein_lem}
we see that the proof goes through by replacing $Z_p$ by $\w Z_p:=(\w B_p-E[\w B_p])/\w s_N$. Finally, using all the analogues of results in~\S\ref{s2} that we mentioned so far allows one to modify the arguments of the proof of Proposition~\ref{prop:moments_stein}
to obtain the following result: \begin{proposition}\label{propwwwmoments_stein} Fix $z\geq 0 $ and a positive integer  $r $. 
Then we have \[ \lim_{N\to \infty}E\l[\l(\frac{\w S_N-\w c_N}{\w s_N}\r)^r   \r] =\mu_r,\] where $\mu_r$ is the $r$-th moment of the standard normal distribution. 
\end{proposition} This concludes the probabilistic part of the proof of Theorem~\ref{thm2}.
The number theoretic   part requires the Fundamental lemma of sieve theory and the following:
\begin{proposition}\label{warcturus}  Keep the setting  of Theorem~\ref{thm2}. Then there exist constants $\delta >1, A>0 
$ that depends on $V$ and $f$ with the following property. Let $Q \in \NN$ with $p \nmid Q$ for all $p \leq A$. 
Then for all $\epsilon>0$ and $Q\leq B^{1/6}$ we have  \[\#\left\{x\in \P^n(\Q):		\begin{array}{l} H(x)\leq B,  f^{-1}(x) \textrm{ smooth}\\
			f^{-1}(x)(\Q_p)=\emptyset \  \forall \  p\mid Q \end{array}	\right\} 		 =   c_n B^{n+1}\prod_{p \mid Q} \sigma_{p}    + O 
			\left(  \frac{\delta^{\omega(Q)} B^{n+1}}{Q\min \{p\mid Q\} }	\right),	\] where the implied constant is independent of $B$ and $Q$. 
			\end{proposition} \begin{proof} By~\cite[Proposition 3.4]{loughrsof} there exist $\alpha >0 , d\in \N $ such that the left hand side is at most 
\[  c_nB^{n+1}  \Big(\prod_{p \mid Q} (\sigma_{p} + \alpha/p^2)\Big) + O\left( (4d)^{\omega(Q)}   	( Q^{2n+1} B + Q B^n(\log B)^{[1/n]} ) \right).	\] 
while, it exceeds a similar quantity with $\alpha $ replaced by $-\alpha$. As shown in~\cite[Lemma 3.7]{loughrsof} we have 
\[ \prod_{p \mid Q} (\sigma_{p} + \alpha/p^2)= \prod_{p \mid Q} \sigma_{p}  + O\left( \frac{(2\alpha d)^{\omega(Q)}}{Q \min\{p: p\mid Q\}}	\right) .\]
This is satisfactory by defining   $\delta=2+\max\{4d, 2\alpha d\}$. Finally,
  \[	  Q^{2n+1} B + Q B^{n} \log B \ll  \frac{B^{n+1}}{Q^2} \leq \frac{B^{n+1}}{Q \min\{p\mid Q\}}\]
owing to 
$Q\leq B^{1/6}$.
 \end{proof}
Our next task is to show that the moments of a truncated version of $\omega_{f,z}$ are asymptotically Gaussian.
For this we shall follow the arguments in \S\ref{s:justif}, where $\Omega_n =\N\cap [1,n]$ is replaced by the uniform discrete probability space
 $$\w \Omega_B=\{x\in \P^n(\Q): H(x) \leq B, f^{-1}(x) \textrm{ smooth}\}$$
for $B>0$.  The condition that $f^{-1}(x) \textrm{ smooth}$ is included in the definition of $\w \Omega_B$ to make $\omega_{f,z}(x)$ well-defined for each $x\in \w \Omega_B$ .
Choosing a polynomial which vanishes on the singular locus of $f$ we see that 
$$\#\{x\in \P^n(\Q): H(x) \leq B, f^{-1}(x) \textrm{ not smooth}\}=O(B^n).$$ Then the standard result 
\[ \#\{x\in \P^n(\Q): H(x) \leq B\}=c_nB^{n+1}+O(B^n(\log B)^{[1/n]}) ,\] where $c_n=2^n/\zeta(n+1)$,  shows that  
\[ \#\w \Omega_B=c_n B^{n+1}+O\l(B^n (\log B)^{[1/n] }\r) .\] We furthermore let for $x\in \P^n(\Q)$, \[
\w \delta_{p,z}(x) :=\begin{cases} \label{wwdeltapz}
1, & \text{ if } f^{-1}(x)(\Q_p)=\emptyset  \text{ and } f^{-1}(x)(\Q_q)\neq \emptyset  \text{ for every prime } q\in (p,p^{\exp(z) }],
\\  0, &  \text{ otherwise.} \end{cases}\] We shall choose any   two functions  $t_0,t_1:(0,\infty )\to (0,\infty )$ satisfying 
\beq{eq:optimcho}{
1<t_0(B)<t_1(B)<B,\lim_{B\to\infty} t_0(B)=\lim_{B\to\infty} t_1(B)=\infty.
} They will be chosen   optimally   later.
The analogue of~\eqref{def:walc} in our setting 
is defined as  
\beq{wwwalc}{\w \omega_{z,B}(x)=\sum_{t_0(B)<p\leq t_1(B)}\w \delta_{p,z}(x).}
We obtain  for $r\in \N$, \beq{eq:wwcanasm}{E_{ x\in \w \Omega_B}\l[\w \omega_{z,B}(x)^r \r] =
\sum_{u=1}^r \sum_{\substack{ r_1,\ldots, r_u \in \N \\ r_1+\ldots +r_u=r }} \frac{r!}{r_1!\cdots r_u!} 
\sum_{\substack{ t_0(B)< p_1 < \cdots <p_u \leq t_1(B)\\ p_{i+1} > p_i^{\exp(z) } \ \forall i}} \hspace{-0,5cm}
E_{ x\in \w \Omega_B}\l[\w \delta_{p_1,z}(x) \cdots \w\delta_{p_u,z}(x) \r] ,} where we added the assumption $p_{i+1} > p_i^{\exp(z) } \ \forall i$ 
since otherwise the expectation in the right hand side vanishes. 

Let us now define the function $m_B:\P^n(\Q)\to \N$ given by \[m_B(x):=\prod_{t_0(B)<p \leq t_1(B) \atop f^{-1}(x)(\Q_p)=\emptyset } p.\]
Letting $\w x$ be the product of all primes $p\leq t_1(B)$ we note that 
$m_B(x) \leq \w x$. Now let $\c P$ be the set of all primes in  $\bigcup_{i=1}^u (p_i, p_i^{\exp(z)}]$  and for $m\in \N$ let 
\[\w a_m:= \begin{cases}  \#\{x\in \w\Omega_B: m_B(x)=m\}/\#\w \Omega_B, & \text{ if } m\equiv 0 \md{p_1\cdots p_u},
\\  0, &  \text{ otherwise.} \end{cases}\]  
 This gives  $$ E_{ x\in \w \Omega_B}\l[\w \delta_{p_1,z}(x) \cdots \w\delta_{p_u,z}(x) \r] 
=
 \sum_{ \substack{  m\leq \w x  \\ p\in \c P \Rightarrow p\nmid m }}\w a_m  .$$ We shall use   Lemma~\ref{lem:fund} with $a_m : =\w a_m  , \kappa=\Delta(f) $,
$$ X=  \prod_{i=1}^u \sigma_{p_i}   ,\ \
 g(d)=\prod_{p\mid d }\sigma_p, \ \ D=B^{1/10}, \ \ 
y= t_1(B)^{2\exp(z)}. $$ 
The assumption $0\leq g(p)<1$ is satisfied here due to $\sigma_p\ll 1/p$ and $p>t_0(B) \to \infty $.
Note that for square-free $d$  that is only divisible by primes in $\c P$ we have 
\[r_d=-g(d)X + \sum_{m\leq \w x \atop m\equiv 0 \md d } \w a_m .\] Assuming that 
\beq{eq:t1t0bnd}{\log t_1(B)=o(\log B)} we see that when  $d\leq D$, one has 
\beq{eq:tlt1}{
dp_1\cdots p_u \leq 
dt_1(B)^u \leq Dt_1(B)^u
=B^{\frac{1}{10}+\frac{\log t_1(B)}{\log B}}
\leq B^{1/6}
} for all large $B$. This allows us to  employ Proposition~\ref{warcturus} with $Q=dp_1\cdots p_u$   to obtain 
\[\sum_{m\leq \w x \atop m\equiv 0 \md d } \hspace{-0,3 cm}
\w a_m =\frac{c_nB^{n+1} X g(d) }{\#\w\Omega_B}
+O\l(\frac{\delta^{u+\omega(d) }B^{n+1}  }{\#\w\Omega_B d p_1^2 p_2 \cdots p_u   }\r)
=Xg(d) +O_u\l(\frac{(\log B)^{[1/n]}}{B}+\frac{\delta^{ \omega(d) }  }{  d p_1^2 \prod_{i=2}^u   p_i   }\r)
,\] where we used $\#\w\Omega_B=c_n B^{n+1}+O(B^n(\log B)^{[1/n]})$ and $X g(d) \ll 1 $. 
The inequality~\eqref{eq:tlt1} shows that 
\[
\frac{  d p_1^2 \prod_{i=2}^u   p_i   }{\delta^{ \omega(d) }  }
\leq \l(d   \prod_{i=1}^u   p_i  \r)^2 \leq B^{1/3} \leq \frac{B}{\log B} 
,\]
hence,  
\[r_d=-
Xg(d)+\sum_{m\leq \w x,   m\equiv 0 \md d } 
\w a_m  \ll
 \frac{\delta^{ \omega(d) }  }{  d p_1^2 \prod_{i=2}^u   p_i   }
.\] This shows that the error term occurring in~\eqref{eq:givv} is   
\[\ll 
\frac{X}{ \mathrm e^{s}}
\prod_{p<y  \atop p\in\c P}(1-\sigma_p)
+ \sum_{d\leq B \atop p\mid d \Rightarrow p\in \c P } \frac{|\mu(d)| \delta^{ \omega(d) }  }{  d p_1^2 \prod_{i=2}^u   p_i  }
\leq 
\frac{X}{ \mathrm e^{s}}+\frac{1}{p_1^2 \prod_{i=2}^u   p_i } \prod_{p\in \c P } \l(1+\frac{1}{p}\r)^\delta   .\]
The product over $p\in \c P$ equals \[\prod_{i=1}^{u} \prod_{p_i < p \leq p_i ^{\exp(z) } } \l(1+\frac{1}{p}\r)^\delta\ll \prod_{i=1}^{u} 
\l(\frac{\log (p_i ^{\exp(z) })}{\log p_i }\r)^\delta  \ll 1 .\] Furthermore,   the estimates $p_1 >t_0(B)$ and  $X\ll 1/(p_1\cdots p_u)  $ 
show that 
\[
\frac{X}{ \mathrm e^{s}}+\frac{1}{p_1^2 \prod_{i=2}^u   p_i } \prod_{p\in \c P } \l(1+\frac{1}{p}\r)^\delta  
 \ll \frac{1}{p_1\cdots p_u } \frac{1}{\min \{ \mathrm e^{s} , t_0(B) \} } .\]   The main term occurring in Lemma~\ref{lem:fund} is   
\[X\prod_{p\in \c P }(1-\sigma_p)=\prod_{i=1}^u\l(\sigma_{p_i} \prod_{p_i< p \leq p_i^{\exp(z) }} (1-\sigma _p ) \r),\]
hence,  the expectation $E_{ x\in \w \Omega_B}\l[\w \delta_{p_1,z}(x) \cdots \w\delta_{p_u,z}(x) \r]$ in the right-hand side of~\eqref{eq:wwcanasm}
equals  \[\prod_{i=1}^u\l(\sigma_{p_i} \prod_{p_i< p \leq p_i^{\exp(z) }} (1-\sigma _p ) \r)+O\l(\frac{1}{p_1\cdots p_u } \frac{1}{\min \{ \mathrm e^{s} , t_0(B) \} }\r) .\] 
 Injecting this into~\eqref{eq:wwcanasm} produces the error term\[\ll_r\frac{1}{\min \{ \mathrm e^{s} , t_0(B) \} }
\sum_{u=1}^r   \sum_{   p_1 < \cdots <p_u \leq t_1(B) }  \frac{1}{p_1\cdots p_u } \ll_r  \frac{(\log \log t_1(B) )^r}{\min \{ \mathrm e^{s} , t_0(B) \} } .\]
Following arguments similar to the ones in the proof of Lemma~\ref{lem:fund} the main term is 
\[ \sum_{u=1}^r \sum_{\substack{ r_1,\ldots, r_u \in \N \\ r_1+\ldots +r_u=r }} \frac{r!}{r_1!\cdots r_u!} 
\sum_{\substack{ t_0(B)< p_1 < \cdots <p_u \leq t_1(B)\\ p_{i+1} > p_i^{\exp(z) } \ \forall i}} 
\prod_{i=1}^u \sigma_{p_i} \prod_{p_i< p \leq p_i^{\exp(z) }} (1-\sigma _p )   =E\l[   T_B^r\r],\] where \[   T_B:=\sum_{t_0(B)< p \leq t_1(B) } \w B_p.\]
We have shown that for all $r\in \N$ one has 
\[\l|E_{ x\in \w \Omega_B}\l[\w \omega_{z,B}(x)^r\r]-E\l[   T_B^r\r]  \r|\ll_r \frac{(\log \log t_1(B) )^r}{\min \{ \mathrm e^{s} , t_0(B) \} } .\]
Noting that $T_B= \w S_{t_1(B)}-\w S_{t_0(B)}$ gives 
\[E\l[   T_B^r\r]=E\l[    \w S_{t_1(B)}^r\r]+O_r\l(\max_{0\leq k \leq r-1 }  E\l[    \w S_{t_0(B)}^{r-k}    \w S_{t_1(B)}^k\r]\r).\] 
and the Cauchy--Schwarz inequality  shows that 
$$
E\l[   T_B^r\r]=E\l[    \w S_{t_1(B)}^r\r]+O_r\l(\max_{0\leq k \leq r-1 }
   E\l[    \w S_{t_0(B)}^{2( r-k ) }   \r]^{1/2}
 E\l[       \w S_{t_1(B)}^{2 k } \r]^{1/2}
 \r).$$   
Since $0\leq \w B_p\leq \w X_p $, we infer that $0\leq \w S_N \leq \sum_{p\leq N } \w X_p$, hence,  
\[ E\l[    \w S_{N}^r\r] \leq E\l[   \l( \sum_{p\leq N } \w X_p \r)^r\r]=
\sum_{u=1}^r \sum_{\substack{ r_1,\ldots, r_u \in \N \\ r_1+\ldots +r_u=r }} \frac{r!}{r_1!\cdots r_u!} 
\sum_{\substack{   p_1 < \cdots <p_u \leq N }} \prod_{i=1}^u E[\w X_p^{r_i} ]  
. \] But $E[\w X_p^{r_i} ] = E[\w X_p ] =\sigma_p\ll 1/p$, hence, $E\l[    \w S_{N}^r\r] \ll (\log \log N)^r $.
Hence, \[\max_{0\leq k \leq r-1 } 
E\l[    \w S_{t_0(B)}^{2( r-k ) }   \r]^{1/2}
E\l[       \w S_{t_1(B)}^{2 k } \r]^{1/2}
\ll
\max_{0\leq k \leq r-1 } (\log \log t_0(B))^{r-k} (\log \log t_1(B) )^k, \] which is $\ll (\log \log t_0(B)) (\log \log t_1(B) )^{r-1}.$ Hence, 
\[\l|E_{ x\in \w \Omega_B}\l[\w \omega_{z,B}(x)^r\r]-E\l[    \w S_{t_1(B)}^r\r]  \r|\ll_r \frac{(\log \log t_1(B) )^r}{
\min \l\{ \frac{\log \log t_1(B) }{\log \log t_0(B)},\mathrm e^{s} , t_0(B) \r\} } .\]
Therefore, 
\[ \l|  E_{ \w \Omega_B}\l[\l(\frac{\w \omega_{z,B}(x)-\w c_{t_1(B)} }{\w s_{t_1(B)} }\r)^k\r]  -   E \l[\l(\frac{S_N-\w c_{t_1(B)} }{\w s_{t_1(B)} }\r)^k\r] \r| 
\]is \[ \ll (\w s_{t_1(B)})^{-k}
\sum_{r=0}^k  \l|\w c_{t_1(B)} \r| ^{k-r}\l( E_{  \w \Omega_B}\l[\w \omega_{z,B} ^r \r]-E\l[S_{\w t_1(B)}^r \r] \r)
,\] which, by~\eqref{wwmeanlem}-\eqref{wmeanvar} is 
\[ \ll (\log \log t_1(B) )^{-r}  \frac{(\log \log t_1(B) )^r}{\min \l\{ \frac{\log \log t_1(B) }{\log \log t_0(B)},\mathrm e^{s} , t_0(B) \r\}  }\ll
\frac{1}{\min \l\{ \frac{\log \log t_1(B) }{\log \log t_0(B)},\mathrm e^{s} , t_0(B) \r\}  } . \]
This vanishes asymptotically as long as we 
assume that 
\beq{eq:furthbnndd}{ 
\log \log t_0(B) =o(\log \log t_1(B)).}
This is due to~\eqref{eq:t1t0bnd} which implies that  
\[
s=\frac{\log D}{\log y  }=\frac{1}{20 \exp(z)} \frac{\log B}{\log t_1(B)}\to+\infty.
\]
We have therefore shown that, subject to~\eqref{eq:optimcho}-\eqref{eq:t1t0bnd}-\eqref{eq:furthbnndd}, one has 
\[E_{ \w \Omega_B}\l[\l(\frac{\w \omega_{z,B}(x)-\w c_{t_1(B)} }{\w s_{t_1(B)} }\r)^k\r] \to \mu_k. \]
The concluding arguments follow those in Lemma~\ref{lem:mnts23}, the only difference being dealing with primes $p\leq t_0(B)$.
Recall from~\cite[Lemma 3.2, part (2)]{loughrsof} that there exists  a constant $A>0$ and a homogeneous $F\in \Z[x_0,\ldots, x_n]$ (both of which depend only on $f$) with the property that for all primes $p$ and $x\in \P^n(\Q)$ with $f^{-1}(x) $ smooth
and $f^{-1}(x) (\Q_p)=\emptyset$, one has $p\mid F(x)$. Then  \[0\leq   \omega_{f, z}(x)- \w \omega_{  z,B}(x) \leq  \sum_{p\leq t_0(B) } 1  
+\sum_{p> t_1(B) \atop f^{-1}(x) (\Q_p)=\emptyset} 1  \leq t_0(B) + \#\{p\mid F(x): p>t_1(B) \} .\]  For $z>1$ and $m \in \mathbb N$ we have $\#\{p\mid m : p>z\} \leq (\log m )/(\log z)$. For $x\in \w \Omega_B$ we have $H(x)\leq B$, thus, $\log |F(x)|\ll \log B$. In particular,  
\[  \omega_{f, z}(x)= \w \omega_{ z,B}(x) +O\l(t_0(B)+\frac{\log B}{\log t_1(B)}\r),\] where the implied constant is independent of $B,z$ and $x$.
Combined with arguments similar to the ones in Lemma~\ref{lem:mnts23} we obtain 
\[\lim_{B\to \infty} \frac{ \max \l\{  \l|   ( \omega_{f, z}(x)-\Delta(f)\mathrm e ^{-z\Delta(f) }\log \log B ) - ( \w \omega_{  z,B}(x)-\w c_{t_1(B)}  )   \r| :x\in \w \Omega_B\r\} } {\sqrt{\log \log B } }=0 ,\]
as long as \beq{eq:megscol}{
t_0(B)=o(\sqrt{\log \log B } )
\ \ \textrm{ and } \ \ 
\frac{\log B}{\log t_1(B)}=o(\sqrt{\log \log B } ).}
The proof of Theorem~\ref{thm2}
concludes by adapting 
the  arguments in \S\ref{prfmainthm} 
to  the current setting. This can be achieved as long as  
  we   assume that 
\beq{eq:alsomnsgtd}{\frac{\log \log t_1(B)}{\log \log B}\to 1 
\ \ \textrm{ and } \ \ 
\log \log B-\log \log t_1(B)=o(  \sqrt{\log \log B} )  
 } and it now remains to find functions $t_0(B)$ and $ t_1(B)$ that satisfy all assumptions~\eqref{eq:optimcho}-\eqref{eq:t1t0bnd}-\eqref{eq:furthbnndd}-\eqref{eq:megscol}-\eqref{eq:alsomnsgtd}. This can be done by  
choosing $t_0(B)$ and $t_1(B)$ so that 
\[t_0(B)=\log \log \log  B
\ \ \textrm{ and } \ \ 
 \frac{ \log t_1(B) }{\log B}=\frac{\log \log \log B}{\sqrt{\log \log B}}
.\]


\begin{thebibliography}{99}

 \bibitem{MR466055}P. Billingsley, {\em On the central limit theorem for the prime divisor functions.} {\em Ann. Probability} {\bf 2} (1974), 749--791.
 
  	  
\bibitem{arXiv:1908.01752} 
S. Chan and P. Koymans and D. Milovic and C. Pagano,  
{\em On the negative Pell equation.} 
 \href{https://arxiv.org/abs/1908.01752}{arXiv:1908.01752}.


\bibitem{MR902529} 
J.-M. de Koninck and J. Galambos, 
{\em The intermediate prime divisors of integers.} {\em Proc. Amer. Math. Soc.} {\bf 101} (1987), 213--216.

 \bibitem{arXiv:1908.01752} 
R. de la Bret\`eche and G. Tenenbaum,  
{\em On the gap distribution of prime factors.}  
 \href{https://arxiv.org/abs/2107.02055}{arXiv:2107.02055}.

\bibitem
{MR560507}
P. D. T. A. Elliott,
{\em Probabilistic number theory. {II}}
Grundlehren der Mathematischen Wissenschaften   {\bf 240},
Springer-Verlag, New York-Berlin, 1980.

 \bibitem
{MR0016078} 
P. Erd\H{o}s, 
{\em Some remarks about additive and multiplicative functions.} {\em Bull. Amer. Math. Soc.} {\bf 52} (1946), 527--537.
   
\bibitem
{MR0002374} 
P. Erd\H{o}s and M. Kac, 
{\em The {G}aussian law of errors in the theory of additive number
              theoretic functions.} {\em Amer. J. Math.} {\bf 62} (1940), 738--742.


\bibitem
{MR1210525} 
J. Galambos,
{\em Extensions of some extremal properties of prime divisors to
              {P}oisson limit theorems.} {\em A tribute to {E}mil {G}rosswald: number theory and related
              analysis.} Contemp. Math. {\bf 143} (1993), 363--369.

 
\bibitem
{MR2507311} 
A. J. Harper, 
{\em Two new proofs of the {E}rd\H{o}s-{K}ac theorem, with bound on the
              rate of convergence, by {S}tein's method for distributional
              approximations.} {\em Math. Proc. Cambridge Philos. Soc.} {\bf 147} (2009), 95--114.

\bibitem{MR2647984} J. B. Friedlander and H. Iwaniec, Opera de cribro. {\em American Mathematical Society Colloquium Publications}, {\bf 57}, (2010), xx+527.
 
  \bibitem{MR2310491} A. Granville, {\em Prime divisors are {P}oisson distributed.} {\em Int. J. Number Theory} {\bf 3} (2007), 1--18.
  
  \bibitem{MR2725505} E. Kowalski and A. Nikeghbali, {\em Mod-{P}oisson convergence in probability and number theory.} {\em Int. Math. Res. Not. IMRN} {\bf 18} (2010), 3549--3587.
 
   \bibitem{MR3568035} D. Loughran and A. Smeets, {\em Fibrations with few rational points.} {\em Geom. Funct. Anal.} {\bf 26} (2016), 1449--1482.

  
\bibitem{loughrsof} 
D. Loughran and E. Sofos, 
{\em An {E}rd\H{o}s-{K}ac law for local solubility in families of varieties.} {\em Selecta Mathematica} (2019), to appear. 
 \href{https://arxiv.org/abs/1711.08396v2}{arXiv:1711.08396} 

  \bibitem{MR96629} A. R\'{e}nyi and P. Tur\'{a}n, {\em On a theorem of {E}rd\H{o}s-{K}ac.} {\em Acta Arith.} {\bf 4} (1958), 71--84.

  \bibitem{MR1075658}  J.-P. Serre, {\em {Sp\'{e}cialisation des \'{e}l\'{e}ments de {${\rm Br}_2({\bf
              Q}(T_1,\cdots,T_n))$}}}. {\em C. R. Acad. Sci. Paris S\'{e}r. I Math.} {\bf 311} (1990), 397--402.
 
    \bibitem{skor} A. N. Skorobogatov, {\em Descent on fibrations over the projective line.} {\em Amer. J. Math.} {\bf 118} (1996), 905--923.

  \bibitem{MR882007} C. Stein, {\em Approximate computation of expectations.} Institute of Mathematical Statistics Lecture Notes-Monograph Series   {\bf 7}, Hayward, CA, 1986.
 
  \bibitem{vaart} A. W. van der Vaart, {\em Asymptotic Statistics}. Camb. Ser. Stat. Probab. Math., Cambridge Univ. Press, 1998.

  \end{thebibliography}
\end{document}